\documentclass{amsart}
\usepackage{amsfonts}

\newtheorem{theorem}{Theorem}
\theoremstyle{plain}

\newtheorem{lemma}{Lemma}

\newtheorem{proposition}{Proposition}

\numberwithin{equation}{section}
\input{tcilatex}

\begin{document}
\title[Some inequalities for differentiable convex functions]{Some
inequalities for differentiable convex functions with applications}
\author{Mevl\"{u}t TUN\c{C}$^{\clubsuit }$}
\address{$^{\clubsuit ,\spadesuit }$Mustafa Kemal University, Faculty of
Science and Arts, Department of Mathematics, 31000, Hatay, Turkey}
\email{mevluttttunc@gmail.com}
\urladdr{}
\thanks{}
\author{Sevil BALGE\c{C}T\.{I}$^{\spadesuit }$}
\email{sevilbalgecti@gmail.com }
\urladdr{}
\thanks{$^{\clubsuit }Corresponding$ $Author$}
\dedicatory{}
\thanks{}

\begin{abstract}
In this paper, the Authors establish a new identity for differentiable
functions. By the well-known H\"{o}lder and power mean inequality, they
obtain some integral inequalities related to the convex functions and apply
these inequalities to special means.
\end{abstract}

\maketitle

\section{Introductions}

A function $f:I\rightarrow
\mathbb{R}
$ is said to be convex on $I$ if inequality%
\begin{equation}
f\left( tx+\left( 1-t\right) y\right) \leq tf\left( x\right) +\left(
1-t\right) f\left( y\right)  \label{101}
\end{equation}%
holds for all $x,y\in I$ and $t\in \left[ 0,1\right] $. We say that $f$ is
concave if $(-f)$ is convex.

Geometrically, this means that if $P,Q$ and $R$ are three distinct points on
the graph of $f$ with $Q$ between $P$ and $R$, then $Q$ is on or below the
chord $PR$.

\begin{theorem}
\textbf{The Hermite-Hadamard inequality:} Let $f:I\subseteq
\mathbb{R}
\rightarrow
\mathbb{R}
$ be a convex function and $a,b\in I$ with $a<b$. The following double
inequality:%
\begin{equation}
f\left( \frac{a+b}{2}\right) \leq \frac{1}{b-a}\int_{a}^{b}f\left( x\right)
dx\leq \frac{f\left( a\right) +f\left( b\right) }{2}  \label{110}
\end{equation}%
is known in the literature as Hadamard's inequality (or Hermite-Hadamard
inequality) for convex functions. If $f$ is a positive concave function,
then the inequality is reversed.
\end{theorem}

In \cite{dra2}, Dragomir and Agarwal\ obtained inequalities for
differentiable convex mappings which are connected to Hadamard's inequality,
as follow:

\begin{theorem}
Let $f:I\subseteq
\mathbb{R}
\rightarrow
\mathbb{R}
$ be a differentiable mapping on $I^{\circ }$, where $a,b\in I$, with $a<b$.
If $\left\vert f^{\prime }\right\vert ^{q}$ is convex on $\left[ a,b\right] $%
, then the following inequality holds:
\begin{equation}
\left\vert \frac{f\left( a\right) +f\left( b\right) }{2}-\frac{1}{b-a}%
\int_{a}^{b}f\left( x\right) dx\right\vert \leq \frac{b-a}{8}\left[
\left\vert f^{\prime }\left( a\right) \right\vert +\left\vert f^{\prime
}\left( b\right) \right\vert \right] .  \label{11}
\end{equation}
\end{theorem}

In \cite{cem}, Pearce and Pe\v{c}ari\'{c} obtained inequalities for
differentiable convex mappings which are connected with Hadamard's
inequality, as follow:

\begin{theorem}
Let $f:I\subseteq
\mathbb{R}
\rightarrow
\mathbb{R}
$ be differentiable mapping on $I^{\circ }$, where $a,b\in I$, with $a<b$.
If $\left\vert f^{\prime }\right\vert ^{q}$ is convex on $\left[ a,b\right] $
for some $q\geq 1$, then the following inequality holds:%
\begin{equation}
\left\vert \frac{f\left( a\right) +f\left( b\right) }{2}-\frac{1}{b-a}%
\int_{a}^{b}f\left( x\right) dx\right\vert \leq \frac{b-a}{4}\left( \frac{%
\left( \left\vert f^{\prime }\left( a\right) \right\vert ^{q}+\left\vert
f^{\prime }\left( b\right) \right\vert ^{q}\right) }{2}\right) ^{\frac{1}{q}%
}.  \label{12}
\end{equation}%
If $\left\vert f^{\prime }\right\vert ^{q}$ is concave on $\left[ a,b\right]
$ for some $q\geq 1$, then
\begin{equation}
\left\vert \frac{f\left( a\right) +f\left( b\right) }{2}-\frac{1}{b-a}%
\int_{a}^{b}f\left( x\right) dx\right\vert \leq \frac{b-a}{4}\left\vert
f^{\prime }\left( \frac{a+b}{2}\right) \right\vert .  \label{13}
\end{equation}
\end{theorem}

In \cite{alo2}, Alomari, Darus and K\i rmac\i\ obtained inequalities for
differentiable $s$-convex and concave mappings which are connected with
Hadamard's inequality, as follow:

\begin{theorem}
Let $f:I\subseteq \left[ 0,\infty \right) \rightarrow
\mathbb{R}
$ be differentiable mapping on $I^{\circ }$ such that $f^{\prime }\in L\left[
a,b\right] ,$\ where $a,b\in I$ with $a<b$. If $\left\vert f^{\prime
}\right\vert ^{q}$, $q\geq 1$ is concave on $\left[ a,b\right] $ for some
fixed $s\in \left( 0,1\right] $, then the following inequality holds:%
\begin{eqnarray}
&&\left\vert \frac{f\left( a\right) +f\left( b\right) }{2}-\frac{1}{b-a}%
\int_{a}^{b}f\left( x\right) dx\right\vert  \label{14} \\
&\leq &\left( \frac{b-a}{4}\right) \left( \frac{q-1}{2q-1}\right) ^{1-\frac{1%
}{q}}\left[ \left\vert f^{\prime }\left( \frac{3a+b}{4}\right) \right\vert
+\left\vert f^{\prime }\left( \frac{a+3b}{4}\right) \right\vert \right] .
\notag
\end{eqnarray}
\end{theorem}

For recent results and generalizations concerning Hadamard's inequality and
concepts of convexity and concavity see \cite{alo2}-\cite{yang} and the
references therein.

Throughout this paper we will use the following notations and conventions.
Let $J=\left[ 0,\infty \right) \subset
\mathbb{R}
=\left( -\infty ,+\infty \right) ,$ and $a,b\in J$ with $0<a<b$ and $%
f^{\prime }\in L\left[ a,b\right] $ and
\begin{eqnarray*}
A\left( a,b\right) &=&\frac{a+b}{2},\text{ }G\left( a,b\right) =\sqrt{ab},%
\text{ }I\left( a,b\right) =\frac{1}{e}\left( \frac{b^{b}}{a^{a}}\right) ^{%
\frac{1}{b-a}}, \\
H\left( a,b\right) &=&\frac{2ab}{a+b},\text{ }L\left( a,b\right) =\frac{b-a}{%
\ln b-\ln a} \\
L_{p}\left( a,b\right) &=&\left( \frac{b^{p+1}-a^{p+1}}{\left( p+1\right)
\left( b-a\right) }\right) ^{1/p},\text{ }a\neq b,\text{ }p\in
\mathbb{R}
,\text{ }p\neq -1,0
\end{eqnarray*}%
be the arithmetic, geometric, identric, harmonic, logarithmic, generalized
logarithmic mean for $a,b>0$ respectively.

The main purpose of this paper is to establish refinements of Hadamard's
inequality for convex functions.

\section{Main Results}

In order to establish our main results, we first establish the following
lemma.

\begin{lemma}
\label{ll}Let $f:J\rightarrow
\mathbb{R}
$ be a differentiable function on $J^{\circ }$. If $f^{\prime }\in L\left[
a,b\right] ,$ then%
\begin{eqnarray*}
&&\frac{1}{b-a}\int_{a}^{b}f\left( x\right) dx+\frac{af\left( b\right)
-bf\left( a\right) }{2\left( b-a\right) }-\frac{1}{2}f\left( \frac{a+b}{2}%
\right) \\
&=&\frac{1}{4}\int_{0}^{1}\left( tb+\left( 1-t\right) a\right) f^{\prime
}\left( \frac{1-t}{2}b+\frac{1+t}{2}a\right) dt \\
&&+\frac{1}{4}\int_{0}^{1}\left( ta+\left( 1-t\right) b\right) f^{\prime
}\left( \frac{1-t}{2}a+\frac{1+t}{2}b\right) dt
\end{eqnarray*}%
for each $t\in \left[ 0,1\right] $ and $a,b\in J.$
\end{lemma}

\begin{proof}
Integrating by parts, we get%
\begin{eqnarray*}
&&\frac{1}{4}\int_{0}^{1}\left( tb+\left( 1-t\right) a\right) f^{\prime
}\left( \frac{1-t}{2}b+\frac{1+t}{2}a\right) dt \\
&&+\frac{1}{4}\int_{0}^{1}\left( ta+\left( 1-t\right) b\right) f^{\prime
}\left( \frac{1-t}{2}a+\frac{1+t}{2}b\right) dt \\
&=&(tb+(1-t)a)\left. \frac{f(\frac{1-t}{2}b+\frac{1+t}{2}a)}{\frac{a-b}{2}}%
\right\vert _{0}^{1}-\int_{0}^{1}\frac{f(\frac{1-t}{2}b+\frac{1+t}{2}a)}{%
\frac{a-b}{2}}\left( b-a\right) dt \\
&&+(ta+(1-t)b)\left. \frac{f(\frac{1-t}{2}a+\frac{1+t}{2}b)}{\frac{b-a}{2}}%
\right\vert _{0}^{1}-\int_{0}^{1}\frac{f(\frac{1-t}{2}a+\frac{1+t}{2}b)}{%
\frac{b-a}{2}}\left( a-b\right) dt \\
&=&\frac{bf(a)-af(\frac{a+b}{2})}{\frac{a-b}{2}}+2\int_{0}^{1}f(\frac{1-t}{2}%
b+\frac{1+t}{2}a)dt \\
&&+\frac{af(b)-bf(\frac{a+b}{2})}{\frac{b-a}{2}}+2\int_{0}^{1}f(\frac{1-t}{2}%
a+\frac{1+t}{2}b)dt \\
&=&\frac{bf(a)-af(\frac{a+b}{2})}{\frac{a-b}{2}}+\frac{4}{b-a}\int_{a}^{%
\frac{a+b}{2}}f(x)dx+\frac{af(b)-bf(\frac{a+b}{2})}{\frac{b-a}{2}}+\frac{4}{%
b-a}\int_{\frac{a+b}{2}}^{b}f(x)dx \\
&=&\frac{1}{b-a}\int_{a}^{b}f\left( x\right) dx+\frac{af\left( b\right)
-bf\left( a\right) }{2\left( b-a\right) }-\frac{1}{2}f\left( \frac{a+b}{2}%
\right)
\end{eqnarray*}
\end{proof}

\begin{theorem}
\label{t1}Let $f:J\rightarrow
\mathbb{R}
$ be a differentiable function on $J^{\circ }.$ If $\left\vert f^{\prime
}\right\vert $\ is convex on $J,$ then%
\begin{eqnarray}
&&\left\vert \frac{1}{b-a}\int_{a}^{b}f\left( x\right) dx+\frac{af\left(
b\right) -bf\left( a\right) }{2\left( b-a\right) }-\frac{1}{2}f\left( \frac{%
a+b}{2}\right) \right\vert  \label{s1} \\
&\leq &\left( \frac{5}{48}a+\frac{7}{48}b\right) \left\vert f^{\prime
}\left( a\right) \right\vert +\left( \frac{7}{48}a+\frac{5}{48}b\right)
\left\vert f^{\prime }\left( b\right) \right\vert  \notag
\end{eqnarray}%
for each $a,b\in J.$
\end{theorem}

\begin{proof}
Using Lemma \ref{ll} and from properties of modulus, and since $\left\vert
f^{\prime }\right\vert $\ is convex on $J,$ then we obtain%
\begin{eqnarray*}
&&\left\vert \frac{1}{b-a}\int_{a}^{b}f\left( x\right) dx+\frac{af\left(
b\right) -bf\left( a\right) }{2\left( b-a\right) }-\frac{1}{2}f\left( \frac{%
a+b}{2}\right) \right\vert \\
&\leq &\frac{1}{4}\int_{0}^{1}\left( tb+\left( 1-t\right) a\right)
\left\vert f^{\prime }\left( \frac{1-t}{2}b+\frac{1+t}{2}a\right)
\right\vert dt \\
&&+\frac{1}{4}\int_{0}^{1}\left( ta+\left( 1-t\right) b\right) \left\vert
f^{\prime }\left( \frac{1-t}{2}a+\frac{1+t}{2}b\right) \right\vert dt \\
&\leq &\frac{1}{4}\int_{0}^{1}\left( tb+\left( 1-t\right) a\right) \left[
\frac{1-t}{2}\left\vert f^{\prime }\left( b\right) \right\vert +\frac{1+t}{2}%
\left\vert f^{\prime }\left( a\right) \right\vert \right] dt \\
&&+\frac{1}{4}\int_{0}^{1}\left( ta+\left( 1-t\right) b\right) \left[ \frac{%
1-t}{2}\left\vert f^{\prime }\left( a\right) \right\vert +\frac{1+t}{2}%
\left\vert f^{\prime }\left( b\right) \right\vert \right] dt \\
&=&\frac{1}{4}\left( \frac{1}{6}a+\frac{1}{12}b\right) \left( \left\vert
f^{\prime }\left( b\right) \right\vert \right) +\frac{1}{4}\allowbreak
\left( \frac{1}{3}a+\frac{5}{12}b\right) \left( \left\vert f^{\prime }\left(
a\right) \right\vert \right) \\
&&+\frac{1}{4}\left( \frac{1}{12}a+\frac{1}{6}b\right) \left( \left\vert
f^{\prime }\left( a\right) \right\vert \right) +\frac{1}{4}\allowbreak
\left( \frac{5}{12}a+\frac{1}{3}b\right) \left( \left\vert f^{\prime }\left(
b\right) \right\vert \right) \\
&=&\allowbreak \left( \frac{5}{48}a+\frac{7}{48}b\right) \left\vert
f^{\prime }\left( a\right) \right\vert +\left( \frac{7}{48}a+\frac{5}{48}%
b\right) \left\vert f^{\prime }\left( b\right) \right\vert .
\end{eqnarray*}%
The proof is completed.
\end{proof}

\begin{proposition}
\label{p1}Let $a,b\in J^{\circ },\ 0<a<b,$ then%
\begin{eqnarray*}
&&\left\vert \frac{1}{L\left( a,b\right) }-\frac{1}{H\left( a,b\right) }-%
\frac{1}{2A\left( a,b\right) }\right\vert \\
&\leq &\left( \frac{5}{48}a+\frac{7}{48}b\right) \frac{1}{a^{2}}+\left(
\frac{7}{48}a+\frac{5}{48}b\right) \frac{1}{b^{2}}
\end{eqnarray*}
\end{proposition}

\begin{proof}
The proof follows from (\ref{s1}) applied\ to the convex function $f\left(
x\right) =1/x$.
\end{proof}

\begin{proposition}
\label{p2}Let $a,b\in J^{\circ },\ 0<a<b,$ then%
\begin{eqnarray*}
&&\left\vert L_{n}^{n}\left( a,b\right) +\frac{\left( n-1\right) G^{2}\left(
a,b\right) L_{n-1}^{n-1}\left( a,b\right) }{2}-\frac{1}{2}A^{n}\left(
a,b\right) \right\vert \\
&\leq &\frac{5n}{24}A\left( a^{n},b^{n}\right) +\frac{7n}{24}A\left(
ba^{n-1},ab^{n-1}\right)
\end{eqnarray*}
\end{proposition}

\begin{proof}
The proof follows from (\ref{s1}) applied\ to the convex function $f\left(
x\right) =x^{n},n\geq 2$.
\end{proof}

\ \ \ \ \ \ \ \ \ \ \ \ \ \ \ \ \ \ \ \ \ \ \ \ \ \ \ \ \ \ \ \ \ \ \ \ \ \
\ \ \ \ \ \ \ \ \ \ \ \ \ \ \ \ \ \ \ \ \ \ \ \ \ \ \ \ \ \ \ \ \ \ \ \ \ \
\ \ \

\begin{proposition}
\label{p3}Let $a,b\in J^{\circ },\ 0<a<b,$ then%
\begin{equation*}
\left\vert -\ln I\left( a,b\right) +\frac{\ln \left( a^{b}/b^{a}\right) }{%
2(b-a)}+\frac{1}{2}\ln A\left( a,b\right) \right\vert \leq \frac{12}{48}+%
\frac{7b}{48a}+\frac{5a}{48b}
\end{equation*}
\end{proposition}

\begin{proof}
The proof follows from (\ref{s1}) applied\ to the convex function $f\left(
x\right) =-\ln x$.
\end{proof}

\begin{theorem}
\label{t2}Let$\ f:J\rightarrow
\mathbb{R}
$ be$\ $a\ differentiable\ function on$\ J^{\circ }$. If\ $\left\vert
f^{\prime }\right\vert ^{q}\ $is\ convex\ on$\ \left[ a,b\right] \ $and$\
q>1\ $with$\ \frac{1}{p}+\frac{1}{q}=1$,$\ $then%
\begin{eqnarray}
&&\left\vert \frac{1}{b-a}\int_{a}^{b}f\left( x\right) dx+\frac{af\left(
b\right) -bf\left( a\right) }{2\left( b-a\right) }-\frac{1}{2}f\left( \frac{%
a+b}{2}\right) \right\vert  \label{s2} \\
&\leq &\frac{1}{4^{1+1/q}}L_{p}\left( a,b\right) \left[ \left[ \left\vert
f^{\prime }\left( b\right) \right\vert ^{q}+3\left\vert f^{\prime }\left(
a\right) \right\vert ^{q}\right] ^{\frac{1}{q}}+\left[ \left\vert f^{\prime
}\left( a\right) \right\vert ^{q}+3\left\vert f^{\prime }\left( b\right)
\right\vert ^{q}\right] ^{\frac{1}{q}}\right]  \notag
\end{eqnarray}
\end{theorem}

\begin{proof}
From\ Lemma\ \ref{ll}\ and\ using\ the\ well-known\ H\"{o}lder\ integral\
inequality, we\ obtain%
\begin{eqnarray*}
&&\left\vert \frac{1}{b-a}\int_{a}^{b}f\left( x\right) dx+\frac{af\left(
b\right) -bf\left( a\right) }{2\left( b-a\right) }-\frac{1}{2}f\left( \frac{%
a+b}{2}\right) \right\vert \\
&\leq &\frac{1}{4}\int_{0}^{1}\left( tb+\left( 1-t\right) a\right)
\left\vert f^{\prime }\left( \frac{1-t}{2}b+\frac{1+t}{2}a\right)
\right\vert dt \\
&&+\frac{1}{4}\int_{0}^{1}\left( ta+\left( 1-t\right) b\right) \left\vert
f^{\prime }\left( \frac{1-t}{2}a+\frac{1+t}{2}b\right) \right\vert dt \\
&\leq &\frac{1}{4}\int_{0}^{1}\left( \left( tb+\left( 1-t\right) a\right)
^{p}dt\right) ^{\frac{1}{p}}\left[ \int_{0}^{1}\left\vert f^{\prime }\left(
\frac{1-t}{2}b+\frac{1+t}{2}a\right) \right\vert ^{q}dt\right] ^{\frac{1}{q}}
\\
&&+\frac{1}{4}\int_{0}^{1}\left( \left( ta+\left( 1-t\right) b\right)
^{p}dt\right) ^{\frac{1}{p}}\left[ \int_{0}^{1}\left\vert f^{\prime }\left(
\frac{1-t}{2}a+\frac{1+t}{2}b\right) \right\vert ^{q}dt\right] ^{\frac{1}{q}}
\\
&\leq &\frac{1}{4}\left( \frac{b^{p+1}-a^{p+1}}{\left( b-a\right) \left(
p+1\right) }\right) ^{\frac{1}{p}}\left[ \left\vert f^{\prime }\left(
b\right) \right\vert ^{q}\int_{0}^{1}\frac{1-t}{2}dt+\left\vert f^{\prime
}\left( a\right) \right\vert ^{q}\int_{0}^{1}\frac{1+t}{2}dt\right] ^{\frac{1%
}{q}} \\
&&+\frac{1}{4}\left( \frac{b^{p+1}-a^{p+1}}{\left( b-a\right) \left(
p+1\right) }\right) ^{\frac{1}{p}}\left[ \left\vert f^{\prime }\left(
a\right) \right\vert ^{q}\int_{0}^{1}\frac{1-t}{2}dt+\left\vert f^{\prime
}\left( b\right) \right\vert ^{q}\int_{0}^{1}\frac{1+t}{2}dt\right] ^{\frac{1%
}{q}} \\
&=&\frac{1}{4^{1+1/q}}L_{p}\left( a,b\right) \left[ \left[ \left\vert
f^{\prime }\left( b\right) \right\vert ^{q}+3\left\vert f^{\prime }\left(
a\right) \right\vert ^{q}\right] ^{\frac{1}{q}}+\left[ \left\vert f^{\prime
}\left( a\right) \right\vert ^{q}+3\left\vert f^{\prime }\left( b\right)
\right\vert ^{q}\right] ^{\frac{1}{q}}\right] .
\end{eqnarray*}%
The proof is completed.
\end{proof}

\begin{proposition}
\label{p4}Let $a,b\in J^{\circ },\ 0<a<b,$ then%
\begin{eqnarray*}
&&\left\vert \frac{1}{L\left( a,b\right) }-\frac{1}{H\left( a,b\right) }-%
\frac{1}{2A\left( a,b\right) }\right\vert \\
&\leq &\frac{1}{4^{1+1/q}}L_{p}\left( a,b\right) \left[ \left[
b^{-2q}+3a^{-2q}\right] ^{\frac{1}{q}}+\left[ a^{-2q}+3b^{-2q}\right] ^{%
\frac{1}{q}}\right]
\end{eqnarray*}
\end{proposition}

\begin{proof}
The proof follows from (\ref{s2}) applied\ to the convex function $f\left(
x\right) =1/x$.
\end{proof}

\begin{proposition}
\label{p5}Let $a,b\in J^{\circ },\ 0<a<b,$ then%
\begin{eqnarray*}
&&\left\vert L_{n}^{n}\left( a,b\right) +\frac{\left( n-1\right) G^{2}\left(
a,b\right) L_{n-1}^{n-1}\left( a,b\right) }{2}-\frac{1}{2}A^{n}\left(
a,b\right) \right\vert \\
&\leq &\frac{1}{4^{1+1/q}}L_{p}\left( a,b\right) \left[ \left[ nb^{\left(
n-1\right) q}+3na^{\left( n-1\right) q}\right] ^{\frac{1}{q}}+\left[
na^{\left( n-1\right) q}+3nb^{\left( n-1\right) q}\right] ^{\frac{1}{q}}%
\right]
\end{eqnarray*}
\end{proposition}

\begin{proof}
The proof follows from (\ref{s2}) applied\ to the convex function $f\left(
x\right) =x^{n},n\geq 2$.
\end{proof}

\begin{proposition}
\label{p6}Let $a,b\in J^{\circ },\ 0<a<b,$ then%
\begin{eqnarray*}
&&\left\vert -\ln I\left( a,b\right) +\frac{\ln \left( a^{b}/b^{a}\right) }{%
2(b-a)}+\frac{1}{2}\ln A\left( a,b\right) \right\vert \\
&\leq &\frac{1}{4^{1+1/q}}L_{p}\left( a,b\right) \left[ \left[ b^{-q}+3a^{-q}%
\right] ^{\frac{1}{q}}+\left[ a^{-q}+3b^{-q}\right] ^{\frac{1}{q}}\right]
\end{eqnarray*}
\end{proposition}

\begin{proof}
The proof follows from (\ref{s2}) applied\ to the convex function $f\left(
x\right) =-\ln x$.
\end{proof}

\begin{theorem}
\label{t3}Let\ $f:J\rightarrow
\mathbb{R}
$\ be\ a\ differentiable\ function on\ $J^{\circ }$. If$\ $\ $\left\vert
f^{\prime }\right\vert ^{q}\ $is$\ $convex\ on$\ \left[ a,b\right] \ $and$\
q\geq 1$, then%
\begin{eqnarray}
&&\left\vert \frac{1}{b-a}\int_{a}^{b}f\left( x\right) dx+\frac{af\left(
b\right) -bf\left( a\right) }{2\left( b-a\right) }-\frac{1}{2}f\left( \frac{%
a+b}{2}\right) \right\vert  \label{s3} \\
&\leq &\frac{A^{1-\frac{1}{q}}\left( a,b\right) }{4\times 12^{\frac{1}{q}}}%
\left\{ \left[ \left\vert f^{\prime }(b)\right\vert ^{q}\left( 2a+b\right)
+\left\vert f^{\prime }(a)\right\vert ^{q}\left( 4a+5b\right) \right] ^{%
\frac{1}{q}}\right.  \notag \\
&&+\left. \left[ \left\vert f^{\prime }(a)\right\vert ^{q}\left( a+2b\right)
+\left\vert f^{\prime }(b)\right\vert ^{q}\left( 5a+4b\right) \right] ^{%
\frac{1}{q}}\right\}  \notag
\end{eqnarray}
\end{theorem}

\begin{proof}
From\ Lemma\ \ref{ll}\ \ and$\ $using\ the\ well-known power mean\
inequality\ and$\ $since$\ \left\vert f^{\prime }\right\vert ^{q}\ $is$\ $%
convex\ on$\ \left[ a,b\right] $, we\ have$\bigskip $%
\begin{eqnarray*}
&&\left\vert \frac{1}{b-a}\int_{a}^{b}f\left( x\right)
dx+\frac{af\left(
b\right) -bf\left( a\right) }{2\left( b-a\right) }-\frac{1}{2}f\left( \frac{%
a+b}{2}\right) \right\vert \\
&\leq &\frac{1}{4}\left( \int_{0}^{1}\left( tb+\left( 1-t\right)
a\right) dt\right) ^{1-\frac{1}{q}}\left[ \int_{0}^{1}\left(
tb+\left( 1-t\right) a\right) \left\vert f^{\prime }\left(
\frac{1-t}{2}b+\frac{1+t}{2}a\right)
\right\vert ^{q}dt\right] ^{\frac{1}{q}} \\
&&+\frac{1}{4}\left( \int_{0}^{1}\left( ta+\left( 1-t\right)
b\right) dt\right) ^{1-\frac{1}{q}}\left[ \int_{0}^{1}\left(
ta+\left( 1-t\right) b\right) \left\vert f^{\prime }\left(
\frac{1-t}{2}a+\frac{1+t}{2}b\right)
\right\vert ^{q}dt\right] ^{\frac{1}{q}} \\
&\leq &\frac{1}{4}\left( \int_{0}^{1}\left( tb+\left( 1-t\right)
a\right) dt\right) ^{1-\frac{1}{q}}\left[ \int_{0}^{1}\left(
tb+\left( 1-t\right)
a\right) \left( \frac{1-t}{2}\left\vert f^{\prime }(b)\right\vert ^{q}+\frac{%
1+t}{2}\left\vert f^{\prime }(a)\right\vert ^{q}\right) dt\right] ^{\frac{1}{%
q}} \\
&&+\frac{1}{4}\left( \int_{0}^{1}\left( ta+\left( 1-t\right)
b\right) dt\right) ^{1-\frac{1}{q}}\left[ \int_{0}^{1}\left(
ta+\left( 1-t\right)
b\right) \left( \frac{1-t}{2}\left\vert f^{\prime }(a)\right\vert ^{q}+\frac{%
1+t}{2}\left\vert f^{\prime }(b)\right\vert ^{q}\right) dt\right] ^{\frac{1}{%
q}} \\
&\leq &\frac{1}{4\times 12^{\frac{1}{q}}}(\frac{a+b}{2})^{1-\frac{1}{q}%
}\left\{ \left[ \left\vert f^{\prime }(b)\right\vert ^{q}\left(
2a+b\right)
+\left\vert f^{\prime }(a)\right\vert ^{q}\left( 4a+5b\right) \right] ^{%
\frac{1}{q}}\right. \\
&&+\left. \left[ \left\vert f^{\prime }(a)\right\vert ^{q}\left(
a+2b\right)
+\left\vert f^{\prime }(b)\right\vert ^{q}\left( 5a+4b\right) \right] ^{%
\frac{1}{q}}\right\}
\end{eqnarray*}%
The proof is completed.
\end{proof}

\begin{proposition}
\label{p7}Let $a,b\in J^{\circ },\ 0<a<b,$ then%
\begin{eqnarray*}
&&\left\vert \frac{1}{L\left( a,b\right) }-\frac{1}{H\left( a,b\right) }-%
\frac{1}{2A\left( a,b\right) }\right\vert \\
&\leq &\frac{A^{1-\frac{1}{q}}\left( a,b\right) }{4\times 12^{\frac{1}{q}}}%
\left\{ \left[ b^{-2q}\left( 2a+b\right) +a^{-2q}\left( 4a+5b\right) \right]
^{\frac{1}{q}}\right. \\
&&+\left. \left[ a^{-2q}\left( a+2b\right) +b^{-2q}\left( 5a+4b\right) %
\right] ^{\frac{1}{q}}\right\}
\end{eqnarray*}
\end{proposition}

\begin{proof}
The proof follows from (\ref{s3}) applied\ to the convex function $f\left(
x\right) =1/x$.
\end{proof}

\begin{proposition}
\label{p8}Let $a,b\in J^{\circ },\ 0<a<b,$ then%
\begin{eqnarray*}
&&\left\vert L_{n}^{n}\left( a,b\right) +\frac{\left( n-1\right) G^{2}\left(
a,b\right) L_{n-1}^{n-1}\left( a,b\right) }{2}-\frac{1}{2}A^{n}\left(
a,b\right) \right\vert \\
&\leq &\frac{A^{1-\frac{1}{q}}\left( a,b\right) }{4\times 12^{\frac{1}{q}}}%
\left\{ \left[ \left( nb^{n-1}\right) ^{q}\left( 2a+b\right) +\left(
na^{n-1}\right) ^{q}\left( 4a+5b\right) \right] ^{\frac{1}{q}}\right. \\
&&+\left. \left[ \left( na^{n-1}\right) ^{q}\left( a+2b\right) +\left(
nb^{n-1}\right) ^{q}\left( 5a+4b\right) \right] ^{\frac{1}{q}}\right\}
\end{eqnarray*}
\end{proposition}

\begin{proof}
The proof follows from (\ref{s3}) applied\ to the convex function $f\left(
x\right) =x^{n},n\geq 2$.
\end{proof}

\begin{proposition}
\label{p9}Let $a,b\in J^{\circ },\ 0<a<b,$ then%
\begin{eqnarray*}
&&\left\vert -\ln I\left( a,b\right) +\frac{\ln \left( a^{b}/b^{a}\right) }{%
2(b-a)}+\frac{1}{2}\ln A\left( a,b\right) \right\vert \\
&\leq &\frac{A^{1-\frac{1}{q}}\left( a,b\right) }{4\times 12^{\frac{1}{q}}}%
\left\{ \left[ b^{-q}\left( 2a+b\right) +a^{-q}\left( 4a+5b\right) \right] ^{%
\frac{1}{q}}\right. \\
&&+\left. \left[ a^{-q}\left( a+2b\right) +b^{-q}\left( 5a+4b\right) \right]
^{\frac{1}{q}}\right\}
\end{eqnarray*}
\end{proposition}

\begin{proof}
The proof follows from (\ref{s3}) applied\ to the convex function $f\left(
x\right) =-\ln x$.
\end{proof}


\begin{thebibliography}{99}
\bibitem{alo2} M. Alomari, M. Darus, U.S. K\i rmac\i , \emph{Some
Inequalities of Hermite-Hadamard type for }$s$\emph{-convex Functions}, Acta
Math. Sci. 31B(4):1643-1652 (2011).

\bibitem{alo} M. Alomari, M. Darus, S.S. Dragomir, \emph{Inequalities of
Hermite Hadamard's Type for Functions Whose Derivatives Absolute Values are
Quasi-convex}, RGMIA \textbf{12} (suppl. 14) (2009).

\bibitem{dra} S.S. Dragomir, \emph{Two mappings in connection to Hadamard's
inequalities}, J. Math. Anal. Appl. \textbf{167} (1992) 49-56.

\bibitem{dra3} S.S. Dragomir, Y.J. Cho, S.S. Kim, \emph{Inequalities of
Hadamard's type for Lipschitzian mappings and their applications}, J. Math.
Anal. Appl. \textbf{245} (2000) 489-501.

\bibitem{dra2} S.S. Dragomir, R.P. Agarwal, \emph{Two inequalities for
differentiable mappings and applications to special means of real numbers
and to trapezoidal formula}, Appl. Math. Lett. \textbf{11} (1998) 91-95.

\bibitem{dr2} S. S. Dragomir and C. E. M. Pearce, \emph{Selected topics on
Hermite-Hadamard inequalities and applications}, RGMIA monographs, Victoria
University, 2000. [Online:\
http://www.staff.vu.edu.au/RGMIA/monographs/hermite-hadamard.html].

\bibitem{had} J. Hadamard, \emph{\'{E}tude sur les propri\'{e}t\'{e}s des
fonctions enti\`{e}res en particulier d'une fonction consid\'{e}r\'{e}e par
Riemann}, J. Math. Pures Appl. \textbf{58} (1893) 171-215.

\bibitem{havva} H. Kavurmac\i , M. Avc\i , M E. \"{O}zdemir, \emph{New
inequalities of Hermite-Hadamard type for convex functions with applications}%
, J. Ineq. Appl. 2011, 2011:86, doi:10.1186/1029-242X-2011-86.

\bibitem{kir2} U.S. K\i rmac\i , M.E. \"{O}zdemir, \emph{On some
inequalities for differentiable mappings and applications to special means
of real numbers and to midpoint formula}, Appl. Math. Comput.\ \textbf{153}
(2004) 361-368.

\bibitem{kir} U.S. K\i rmac\i , \emph{Inequalities for differentiable
mappings and applicatios to special means of real numbers to midpoint formula%
}, Appl. Math. Comput. \textbf{147} (2004) 137-146.

\bibitem{mit} D.S. Mitrinovi\'{c}, I.B. Lackovi\'{c}, \emph{Hermite and
convexity}, Aequationes Math. \textbf{28} (1985) 229-232.

\bibitem{mit2} D. S. Mitrinovi\'{c}, J. Pe\v{c}ari\'{c}, and A.M. Fink,\
\emph{Classical and new inequalities in analysis}, KluwerAcademic,
Dordrecht, 1993.

\bibitem{oz} M.E. \"{O}zdemir, \emph{A theorem on mappings with bounded
derivatives with applications to quadrature rules and means}, Appl. Math.
Comput. \textbf{138} (2003) 425-434.

\bibitem{pec} J. E. Pe\v{c}ari\'{c}, F. Proschan and Y. L. Tong, \emph{%
Convex Functions, Partial Orderings, and Statistical Applications}, Academic
Press Inc., 1992.

\bibitem{cem} C.E.M. Pearce, J. Pe\v{c}ari\'{c}, \emph{Inequalities for
differentiable mappings with application to special means and quadrature
formula}, Appl. Math. Lett. \textbf{13} (2000) 51-55.

\bibitem{tunc} M. Tun\c{c}, \emph{On Some New Inequalities for Convex
Functions}, Turk. J. Math., \textbf{36}, (2012) 245-251.

\bibitem{tunc2} M. Tun\c{c}, \emph{On Some Hermite-Hadamard Type
Inequalities For Certain Convex Functions}, Proceedings of The Romanian
Academy, Series A, \textbf{15}(1) (2014) 3-10.

\bibitem{tunc3} M. Tun\c{c}, \emph{Some Hadamard like inequalities via
convex and s-convex functions and their applications for special means},
Mediterranean Journal of Mathematics, Accepted. Doi:
10.1007/s00009-013-0373-y

\bibitem{tunc4} M. Tun\c{c}, S.U. K\i rmac\i , \emph{New Inequalities for
Convex Functions}, Erzincan \"{U}niversitesi Fen Bilimleri Enstit\"{u}s\"{u}
Dergisi, \textbf{3}(1), (2010) 89-99.

\bibitem{xi} B.-Y. Xi and F. Qi, \emph{Some Integral Inequalities of
Hermite-Hadamard Type for Convex Functions with Applications to Means},
Journal of Function Spaces and Appl., Volume 2012, Article ID 980438, 14 p.,
doi:10.1155/2012/980438.

\bibitem{yang} G.S. Yang, D.Y. Hwang, K.L. Tseng, \emph{Some inequalities
for differentiable convex and concave mappings}, Comput. Math. Appl. \textbf{%
47} (2004) 207-216.
\end{thebibliography}
\end{document}